\documentclass[english,10pt]{article}

\usepackage{latexsym}
\usepackage{epsf}
\usepackage[english]{babel}
\usepackage{amsmath,amsfonts,amsthm,amssymb}
\usepackage{varioref, enumerate,multirow}
\newtheorem{thm}{Theorem}[section]    % Kursiivilla

\newtheorem{pro}[thm]{Proposition}
\newtheorem{cor}[thm]{Corollary}
\newtheorem{lem}[thm]{Lemma}

\theoremstyle{definition}           % Normaalilla tekstillä 
\newtheorem*{rem}{Remark}

\numberwithin{equation}{section}

\newcommand{\eps}{\varepsilon}
\newcommand{\C}{c_{d}} 
\newcommand{\N}{\mathbb{N}} 
\newcommand{\R}{\mathbb{R}}
\newcommand{\dist}{\operatorname{dist}}

\newcommand{\essinf}{\operatornamewithlimits{ess\, inf}}

\def\vint_#1{\mathchoice
          {\mathop{\vrule width 6pt height 3 pt depth -2.5pt
                  \kern -8pt \intop}\nolimits_{\hspace{-1ex}#1}}%
          {\mathop{\vrule width 5pt height 3 pt depth -2.6pt
                  \kern -6pt \intop}\nolimits_{\hspace{-1ex}#1}}%
          {\mathop{\vrule width 5pt height 3 pt depth -2.6pt
                  \kern -6pt \intop}\nolimits_{\hspace{-1ex}#1}}%
          {\mathop{\vrule width 5pt height 3 pt depth -2.6pt
                  \kern -6pt \intop}\nolimits_{\hspace{-1ex}#1}}}

\begin{document}

\title{The Gehring lemma in metric spaces}
\author{Outi Elina Maasalo}
\date{21th March 2006\footnote{Revised 15th Jan 2008}}

\maketitle 

\begin{abstract}
We present a proof for the Gehring lemma in a metric measure space endowed with a doubling measure. As an application we show the self--improving property of Muckenhoupt weights.
\end{abstract}

\selectlanguage{english}
\section{Introduction}

The following self--improving property of the reverse H\"older inequality is a result due to Gehring \cite{ge}. Assume that $f$ is a non--negative locally integrable function and $1<p<\infty$. If there is a constant $c$ such that the inequality
\begin{equation}
\label{eka}
\left(\vint_{B}f^pdx \right)^{1/p} \leq c \vint_{B}fdx
\end{equation}
holds for all balls $B$ of $\mathbb{R}^n$, then there exists $\eps>0$ such that
\begin{equation}
\label{toka}
\left(\vint_{B}f^{p+\eps}dx \right)^{1/{p+\eps}} \leq c \vint_{B}fdx
\end{equation}
for some other constant $c$. The theorem remains true also in metric spaces. However, the proof seems to be slightly difficult to find in the literature. 

The subject has been studied for example by Fiorenza \cite{fi} as well as D'Apuzzo and Sbordone \cite{dap-sb}, \cite{sb}. Gianazza \cite{gia} shows that if a function satisfies \eqref{eka}, then there exists $\eps >0$ such that
\begin{equation}
\label{kolmas}
\left(\vint_{X}f^{p+\eps}d\mu \right)^{1/{p+\eps}} \leq c \vint_{X}fd\mu
\end{equation}
for some constant $c$. The result is obtained in a space of homogeneous type, with the assumption that $0<\mu(X)<\infty$. In this paper, our purpose is to present a transparent proof for a version of the Gehring lemma in a metric space that has the annular decay property and that supports a doubling measure.  

Also Kinnunen examines various minimal, maximal and reverse H\"older inequalities in \cite{juha} and \cite{juha-vaitos}. Str\"omberg and Torchinsky prove Gehring's result under the additional assumption that the measure of a ball depends continuously on its radius, see \cite{st-to}. Zatorska--Goldstein \cite{anna} proves a version of the lemma, where on the right--hand side there is a ball with a bigger radius. 

We present a proof of the Gehring lemma in a doubling metric measure space. In addition, we assume that the space has the annular decay property, see Section 2. This is true for example in length spaces, i.e. metric spaces in which the distance between any pair of points is equal to the infimum of the length of rectifiable paths joining them. Our method is classical and intends to be as transparent as possible. In particular, we obtain the result for balls in the sense of \eqref{toka} in the metric setting instead of \eqref{kolmas}. The proof is based on a Calder\'on--Zygmund type argument.

The Gehring lemma has a number of possible applications. As an example we show that the Muckenhoupt class is an open ended condition. The proof is classical and holds without the assumption of annular decay property. Moreover, the Gehring lemma can be applied for example to prove higher integrability of the volume derivative, also known as the Jacobian, of a quasisymmetric mapping, see \cite{he-ko}.

\section{General Assumptions}
\label{preli}
Let $(X,d,\mu)$ be a metric measure space equipped with a Borel regular measure $\mu$ such that the measure of every nonempty open set is positive and that the measure of every bounded set is finite.

Our notation is standard. We assume that a ball $B$ in $X$ comes always with a fixed centre and radius, i.e. $B=B(x,r) =\{ y\in X \colon d(x,y) <r\}$ with $0<r<\infty$. We denote
$$
u_B=\vint_Bud\mu=\frac{1}{\mu(B)}\int_Bud\mu,
$$
and when there is no possibility for confusion we denote $kB$ the ball $B(x,kr)$. 
 We assume in addition that $\mu$ is \emph{doubling} i.e. there exists a constant $\C$ such that
$$
\mu(B(x,2r))\leq \C\mu(B(x,r))
$$
for all balls $B$ in $X$. We refer to this property by calling $(X,d,\mu)$ a doubling metric measure space and denote it briefly $X$.  This is different from the concept of \emph{doubling space}. The latter is a property of the metric space $(X,d)$, where all balls can be covered by a constant number of balls with radius half of the radius of the original ball. A doubling metric measure space is always doubling as a metric space. 

A good reference for the basic properties of a doubling metric measure space is \cite{he}. In particular, we will need two elementary facts. Consider a ball containing disjoint balls such that their radii are bounded below. In a doubling space the number of these balls is bounded. Secondly, the doubling property of $\mu$ implies that for all pairs of radii $0<r\leq R$ the inequality
$$
\frac{\mu(B(x,R))}{\mu(B(x,r))}\leq \C\left(\frac{R}{r}\right)^Q
$$
holds true for all $x\in X$. Here $Q=\log_2\C$ is called the \emph{doubling dimension} of $(X,d,\mu)$.

Given $0<\alpha\leq 1$ and a metric space $(X,d,\mu)$ with a doubling $\mu$, we say that the space satisfies the $\alpha$\textit{--annular decay property} if there exists a constant $c\geq 1$, such that
\begin{equation*}
\mu(B(x,r)\setminus B(x,(1-\delta)r))\leq c\delta^{\alpha}\mu(B(x,r))
\end{equation*}
for all $x\in X$, $r>0$ and $0<\delta <1$. We omit $\alpha$ in the notation if we do not care about its value. See \cite{bu} for further information on spaces that satisfy the annular decay property.
 
Throughout the paper, constants are denoted $c$ and they may not be the same everywhere. However, if not otherwise mentioned, they depend only on fixed constants such as those associated with the structure of the space, the doubling constant etc.

\section{Gehring lemma}
\label{gehring-section}

The following theorem is our main result.

\begin{thm}[Gehring lemma]
\label{gehring-theorem}
Consider $(X,d,\mu)$, where $\mu$ is doubling.  Let $1<p<\infty$ and $f\in L^1_{loc}(X)$ be non--negative. If there exists a constant $c$ such that $f$ satisfies the reverse H\"older inequality
\begin{equation}
\label{rhi}
\left(\vint_{B}f^pd\mu \right)^{1/p} \leq c \vint_{B}fd\mu
\end{equation}
for all balls $B$ of $X$, then there exists $q>p$ such that
\begin{equation}
\label{improved-rhi}
\left(\vint_{B}f^qd\mu \right)^{1/q} \leq c_q \vint_{2B}fd\mu
\end{equation}
for all balls $B$ of $X$. The constant $c_q$ as well as $q$ depend only on the doubling constant, $p$, and on the constant in \eqref{rhi}.
\end{thm}

In metric spaces with a priori more geometrical structure this can be improved.

\begin{cor}
\label{gehring-corollary}
In addition to the assumptions in Theorem \ref{gehring-theorem}, suppose $X$ satisfies the annular decay property. Then the measure induced by $f$ is doubling and \eqref{rhi} implies
\begin{equation}
\label{better-improved}
\left(\vint_{B}f^qd\mu \right)^{1/q} \leq c_q \vint_{B}fd\mu
\end{equation}
for all balls $B$ of $X$. The constant $c_q$ as well as $q$ depend only on the doubling constant, $p$, and on the constant in \eqref{rhi}.
\end{cor}

The following is a standard iteration lemma, see \cite{gi}. 

\begin{lem} 
\label{l:g-iterati}
Let $Z:[R_1,R_2]\subset \R\to [0,\infty)$ be a bounded non--negative function. Suppose that
for all $\rho,r$ such that $R_1\leq \rho<r\leq R_2$ 
\begin{equation}
Z(\rho)\leq \big(A(r-\rho)^{-\alpha}+B(r-\rho)^{-\beta}+C\big) +\theta Z(r)
\end{equation}
holds true for some constants $A,B,C\geq 0$, $\alpha>\beta>0$ and $0\leq \theta<1$. Then 
\begin{equation}
Z(R_1)\leq c(\alpha,\theta)\big(A(R_2-R_1)^{-\alpha}+B(R_2-R_1)^{-\beta}+C\big).
\end{equation}
\end{lem}

Lemma \ref{l:g-iterati} is needed in the proof of our first key lemma:

\begin{lem}
\label{l:g-2iterat}
Let $R>0,$ $q>1$, $k>1$ and $f\in L^q_{loc}(X)$ non--negative. There exists $\eps>0$ such that, if for all $0<r\leq R$ and for a constant $c$
\begin{equation}
\label{e:g-suppose}
\vint_{B(x,r)}f^qd\mu \leq \eps \vint_{B(x,kr)}f^qd\mu + c\left(\vint_{B(x,kr)}fd\mu \right)^q
\end{equation}
holds, then
\begin{equation}
\label{get}
\vint_{B(x,R)}f^q d\mu \leq c\left(\vint_{B(x,2R)}fd\mu \right)^q.
\end{equation}
The constant in \eqref{get} depends on $k$, on the doubling constant and on the constant in \eqref{e:g-suppose}. 
\end{lem}

\begin{proof}
Fix $R>0$ and choose  $r,\rho>0$ such that $R\leq \rho < r\leq 2R$. Set $\tilde{r}=(r-\rho)/k$. Now
$$
B(x,\rho)\subset \bigcup_{y\in B(x,\rho)}B(y,\tilde{r}/5)
$$
and by the Vitali covering theorem there exist disjoint balls $\{B(x_i,\tilde{r}/5)\}_{i=1}^{\infty}$ such that $x_i\in B(x,\rho)$ and
$$
B(x,\rho)\subset \bigcup_{i}B(x_i,\tilde{r}).
$$
These balls can be chosen so that
\begin{equation}
\label{finite-intersection}
\sum_{i}\chi_{B(x_{i},k\tilde{r})} \leq M
\end{equation}
for some constant $M<\infty$. This follows from the doubling property of the space. Indeed, assume that $y$ belongs to $N$ balls $B(x_i,k\tilde{r})$. Clearly
$$
B(x_i,k\tilde{r})\subset B(y,2k\tilde{r})\subset B(y,2R). 
$$
Remember that $\tilde{r}$ and $R$ are fixed and choose $K=20R/\tilde{r}$. Now there are $N$ disjoint balls with radius $\tilde{r}/5\geq 2R/K$ included in a fixed ball $B(y, 2R)$. Since the space is doubling, we must have $N\leq M(K)$. The inequality \eqref{finite-intersection} follows.

Observe then that by the doubling property and the construction of the balls $\{B(x_i,\tilde{r})\}_i$ we have
\begin{equation*}
\begin{split}
\sum_i\mu(B(x_i,\tilde{r}))&\leq  c \sum_i\mu(B(x_i,\tilde{r}/5)) =  c\mu(\cup_i B(x_i,\tilde{r}/5)) \\
&\leq c\mu(B(x,r)) \leq c\left(\frac{r}{\rho}\right)^Q \mu(B(x,\rho)). 
\end{split}
\end{equation*}
On the other hand, $B(x,\rho)\subset B(x_i, 2k\rho)$, so that
\begin{equation*}
\begin{split}
\mu(B(x,\rho)) &\leq \mu(B(x_i, 2k\rho)) \leq c\left(\frac{2k\rho}{\tilde{r}}\right)^Q \mu(B(x_i,\tilde{r})) \\
&=c  \left(\frac{\rho}{r-\rho}\right)^Q \mu(B(x_i,\tilde{r})).
\end{split}
\end{equation*}
Combining these two inequalities implies 
\begin{equation*}
\begin{split}
\mu(B(x,\rho))&\geq c \left(\frac{r}{\rho}\right)^{-Q}\sum_i \mu(B(x_i,\tilde{r})) \\
&\geq  c\left(\frac{r}{\rho}\right)^{-Q}\left(\frac{\rho}{r-\rho}\right)^{-Q} \sum_i\mu(B(x,\rho)).
\end{split}
\end{equation*}
And as a consequence
$$
\# \{B(x_i,\tilde{r})\} \leq c \left(\frac{r}{\rho}\right)^{Q}\left(\frac{\rho}{r-\rho}\right)^{Q}, %= c \left(\frac{r}{r-\rho}\right)^{Q},
$$
i.e. the number of balls $B(x_i, \tilde{r})$ is at most $c\big(r/(r-\rho)\big)^Q$, where $c$ depends only on $k$, on the doubling constant and $Q=\log_2\C$.

Observe that \eqref{e:g-suppose} holds true for $\tilde{r}$, so that 
\begin{align}
\int_{B(x_{i},\tilde{r})}f^q d\mu &\leq  \eps \frac{\mu(B(x_{i},\tilde{r}))}{\mu(B(x_{i},k\tilde{r}))} \int_{B(x_{i},k\tilde{r})}f^q d\mu
 \nonumber\\
&  \quad+c\frac{\mu(B(x_{i},\tilde{r}))}{\mu(B(x_{i},k\tilde{r}))^q} \left(\int_{B(x_{i},k\tilde{r})}f d\mu\right)^q \nonumber\\
&\leq  \eps \int_{B(x_{i},k\tilde{r})}f^q d\mu + c\mu(B(x_i,\tilde{r}))^{1-q}\left(\int_{B(x_{i},k\tilde{r})}f d\mu\right)^q. 
\label{e:g-2suppos}
\end{align}
We note that
\begin{equation*}
\frac{\mu(B(x,r))}{\mu(B(x_{i},\tilde{r}))}\leq \frac{\mu(B(x_i,2r))}{\mu(B(x_{i},\tilde{r}))} 
\leq  \C\left( \frac{2r}{\tilde{r}}\right)^Q %\leq \C\left( \frac{2kr}{\tilde{r}}\right)^Q \\
%&= & \C^2\left( \frac{kr}{\tilde{r}}\right)^Q 
\leq   c \left( \frac{r}{r-\rho}\right)^Q, 
\end{equation*}
from which it follows that
%$$
%\frac{1}{\mu(B(x_{i},\tilde{r}))} \leq c \left( \frac{r}{r-\rho}\right)^Q\frac{1}{\mu(B(x,r))} 
%$$
%and
$$
\mu(B(x_{i},\tilde{r}))^{1-q} \leq c \left( \frac{r}{r-\rho}\right)^{Q(q-1)}\mu(B(x,r))^{1-q}.
$$
Together with \eqref{e:g-2suppos} this implies 
\begin{multline}
\label{e:g-tosum}
\int_{B(x_{i},\tilde{r})}f^q d\mu \leq \eps \int_{B(x_{i},k\tilde{r})}f^q d\mu
\\+ c\left( \frac{r}{r-\rho}\right)^{Q(q-1)}\mu(B(x,r))^{1-q} \left(\int_{B(x_{i},k\tilde{r})}f d\mu\right)^q.
\end{multline}
Since $B(x,\rho)\subset \cup_iB(x_i,\tilde{r})$, summing over $i$ in \eqref{e:g-tosum} gives
\begin{equation*}
\begin{split}
\int_{B(x,\rho)}f^qd\mu &\leq \sum_i\int_{B(x_{i},\tilde{r})}f^qd\mu \\
&\leq  \eps\sum_i \int_{B(x_{i},k\tilde{r})}f^q d\mu \\
&  \quad + c\left( \frac{r}{r-\rho}\right)^{Q(q-1)}\mu(B(x,r))^{1-q}\sum_i \left(\int_{B(x_{i},k\tilde{r})}f d\mu\right)^q \\
&\leq  \eps M \int_{B(x,r)}f^qd\mu \\
&  \quad + c\left( \frac{r}{r-\rho}\right)^{Q(q-1)}\mu(B(x,r))^{1-q}\left(\frac{r}{r-\rho}\right)^Q \left(\int_{B(x,r)}f d\mu\right)^q \\
&= \eps M \int_{B(x,r)}f^qd\mu +  c\left( \frac{r}{r-\rho}\right)^{Qq}\mu(B(x,r))^{1-q} \left(\int_{B(x,r)}f d\mu\right)^q.
\end{split}
\end{equation*}
Finally, remember that $R\leq \rho<r\leq 2R$, so that
\begin{multline*}
\int_{B(x,\rho)}f^qd\mu \leq \eps M \int_{B(x,r)}f^qd\mu \\
+  cR^{Qq}(r-\rho)^{-Qq}\mu(B(x,r))^{1-q} \left(\int_{B(x,r)}f d\mu\right)^q
\end{multline*}
and furthermore
\begin{multline}
\label{e:g-finally}
\vint_{B(x,\rho)}f^qd\mu \leq \eps c \vint_{B(x,r)}f^qd\mu \\
+ cR^{Qq}(r-\rho)^{-Qq}\left(\vint_{B(x,2R)}f d\mu\right)^q.
\end{multline}
We are able to iterate this. In Lemma \ref{l:g-iterati} set 
$$
Z(\rho):=\vint_{B(x,\rho)}f^qd\mu,
$$
so that $Z$ is bounded on $[R,2R]$. Set also $R_1=R$, $R_2=2R$, $\alpha =Qq$ and  
$$
A=cR^{Qq}\left(\vint_{B(x,2R)}fd\mu\right)^q\geq 0,
$$
where $c$ is the constant in \eqref{e:g-finally}. Putting $\theta=c\eps$ and choosing $\eps$ so small that  $c\eps<1$, \eqref{e:g-finally} satisfies the assumptions of Lemma \ref{l:g-iterati} with $B=C=0$. This yields $Z(R)\leq cA(2R-R)^{-Qq}$, that is 
\begin{equation*}
\begin{split}
\vint_{B(x,R)} |f|^qd\mu&\leq  cR^{Qq}(cR-R)^{-Qq}\left(\vint_{B(x,2R)}fd\mu\right)^q \\
&= c\left(\vint_{B(x,2R)}fd\mu\right)^q.
\end{split}
\end{equation*}
\end{proof}

In the following, we consider the Hardy--Littlewood maximal function restricted to a fixed ball $100B_0$, that is
$$
Mf(x)= \sup_{\substack{B\ni x \\ B\subset 100B_0}}\vint_{B}|f|d\mu.
$$
Clearly the coefficient 100 can be replaced by any other sufficiently big constant. The role of this constant is setting a playground large enough to assure that all balls we are dealing with stay inside this fixed ball. 

\begin{lem}
\label{weaktype-lemma}
Let $f$ be a non--negative function in $L^1_{loc}(X)$ and suppose that is satisfies the reverse H\"older inequality \eqref{rhi}. Then for all balls $B_0$ in $X$ 
\begin{equation}
\label{weaktype-equation}
\int_{\{x\in B_0\, : \, Mf(x)>\lambda \}} f^pd\mu\leq c \lambda^p \mu(\{x\in 100B_0 \, :\, Mf(x)>\lambda\}),
\end{equation}
for all $\lambda > \essinf_{B_0}Mf$ with some constant depending only on $p$, the doubling constant and on the constant in \eqref{rhi}.
\end{lem}

\begin{proof}
Let us fix a ball $B_0$ with radius $r_0>0$. 
We denote $\{x\in X\, :\, Mf(x)>\lambda\}$ briefly by  $\{Mf>\lambda\}$. Let $\lambda>\essinf_{B_0}Mf$. Now there exists $x\in B_0$ so that $Mf(x)\leq \lambda$. This implies that $B_0\cap\{Mf\leq \lambda\} \neq \emptyset$. For every $x\in B_0\cap\{Mf>\lambda\}$, set 
$$
r_x=\dist(x,100B_0\setminus \{Mf>\lambda\}),
$$ 
so that $B(x,r_x)\subset 100B_0$. Remark that the radii $r_x$ are uniformly bounded by $2r_0$. 
%Indeed, suppose that $R>2r_0$ and let $x\in B_0$. Now there exists $y\in B_0\cap\{Mf\leq \lambda\}$, so  $y\in B(x,R)$ since $B_0 \subset B(x,R)$. This yields
%$$
%\vint_{B(x,R)}fd\mu \leq \sup_{\substack{B\ni y \\ B\subset 100B_0}}\vint_Bfd\mu =Mf(y)\leq \lambda,
%$$
%and moreover
%$$
%\sup_{\substack{B\ni x \\ B\subset 100B_0}}\vint_Bfd\mu\leq \lambda.
%$$
%But this means that $x\in B_0\cap\{Mf\leq\lambda\}$, so we must have $r_x\leq 2r_0$.

In the consequence of the Vitali covering theorem there are disjoint balls \\ $\{B(x_i, r_{x_{i}})\}_{i=1}^{\infty}$ such that 
$$
B_0\cap\{Mf>\lambda\} \subset \bigcup_{i}5B_i,
$$
where we denote $B_i=B(x_i,r_i)$. Both $B_i\subset 100B_0$ and $5B_i\subset 100B_0$ for all $i=1,2,\ldots$, so they are still balls of $(X,d)$. Furthermore, $5B_i\cap \{Mf\leq \lambda\}\ne \emptyset$ for all $i=1,2,\ldots$ so that
\begin{equation}
\label{e:g-meanbou}
\vint_{5B_i}fd\mu\leq \lambda
\end{equation}
for all $i=1,2,\ldots$. We can now estimate the integral on the left side in \eqref{weaktype-equation}. A standard estimation shows that
\begin{equation*}
\begin{split}
\int_{B_0\cap\{Mf>\lambda\}}f^pd\mu& \leq\int_{\cup_i5B_{i}}f^pd\mu
\leq  \sum_i\int_{5B_i}f^pd\mu \\
& = \sum_i\mu(5B_i)\vint_{5B_i}f^pd\mu \leq  c^p\sum_i\mu(5B_i)\left(\vint_{5B_i}fd\mu\right)^p \\ 
&\leq c^p\lambda^p\sum_i\mu(5B_i), 
\end{split}
\end{equation*}
where the second last inequality follows from the reverse H\"older inequality and the last from \eqref{e:g-meanbou}. Since $\mu$ is doubling and the balls $B_i$ are disjoint we get
$$
\sum_i\mu(5B_i)\leq c\sum_i\mu(B_i) = c\mu(\cup_i B_i).
$$ 
By the definition of  $B_i \subset 100B_0\cap \{Mf>\lambda\}$ for all $i=1,2,\ldots$. Therefore
\begin{equation*}
\int_{B_0\cap\{Mf>\lambda\}}f^pd\mu \leq c\lambda^p\mu(\cup_i B_i)
\leq c\lambda^p\mu(100B_0\cap\{Mf>\lambda \})
\end{equation*}
for all $\lambda>\essinf_{B_{0}}Mf$.
\end{proof}

\begin{rem}
Note that $\essinf_{B_0}Mf \neq \infty$.
\end{rem}
Indeed, in the well known weak type estimate for locally integrable functions
$$
\mu(B_0\cap \{Mf>\lambda\}) \leq \frac{c}{\lambda} \int_{100B_0}f d\mu,
$$
the right--hand side tends to zero when $\lambda \to \infty$. The constant $c$ depends only on the doubling constant $\C$. We can thus choose $0<\lambda_0<\infty$ so that
$$
\frac{c}{\lambda_0}\int_{100B_0}fd\mu \leq \frac{1}{2}\mu(B_0).%\leq \frac{1}{2}\mu(100B_0).
$$ 
%It follows that
%$$
%\lambda_0\geq c\vint_{100B_0}fd\mu.
%$$
As a consequence,
\begin{equation*}
\begin{split}
\mu(B_0\cap\{&Mf\leq \lambda_0 \})= \mu(B_0) -\mu(B_0\cap\{Mf> \lambda_0 \}) \\
&\geq  \mu(B_0)-\frac{c}{\lambda_0}\int_{100B_0}fd\mu 
%&\geq &\mu(B_0) -\frac{1}{2}\mu(B_0) \\
 \geq  \frac{1}{2}\mu(B_0).
\end{split}
\end{equation*}
This leads to $\essinf_{B_{0}}Mf\leq \lambda_0$, for if $\essinf_{B_{0}}Mf> \lambda_0$, then $Mf(x)>\lambda_0$ for almost every $x\in B_0$. This impossible since 
\[
\mu(B_0\cap\{Mf\leq \lambda_0 \})\geq \frac{1}{2}\mu(B_0).
\]

\medskip

For the reader's convenience, we present here one technical part of our proof as a separate lemma.
\begin{lem}
\label{technical-lemma}
Let $1<q<\infty$ and let $f\in L_{loc}^q(X)$ be non--negative. Suppose in addition that $f$ satisfies the reverse H\"older inequality. Then for every ball $B_0$ in $X$ and for all $1<p<q$
\begin{equation}
\label{technical-estimate}
 \int_{B_0\cap \{Mf>\alpha\}}f^qd\mu\leq  c\alpha^q\mu(100B_0\cap \{Mf>\alpha\}) + c\frac{q-p}{q}\int_{100B_0}(Mf)^qd\mu,
\end{equation}
where $\alpha=\essinf_{B_0}Mf$ and $c$ depends on $p$, the doubling constant and on the constant in \eqref{rhi}. 
\end{lem}
\begin{proof}
Fix a ball $B_0$ in $X$. Let $\alpha=\essinf_{B_{0}}Mf$, so that $Mf\geq\alpha$ $\mu$--a.e. on $B_0$. Set $d\nu=f^pd\mu$. Now
\begin{equation*}
 \int_{B_{0}\cap \{Mf>\alpha\}}f^qd\mu = \int_{B_{0}\cap \{Mf>\alpha\}}f^{q-p}f^pd\mu
%&\leq \int_{\{Mf>\alpha\}}(Mf)^{q-p}f^pd\mu \\
\leq \int_{\{ Mf>\alpha\}}(Mf)^{q-p}d\nu. 
\end{equation*}
However, for every positive measure and measurable non--negative function $g$ and a measurable set $E$, we have
$$
\int_{E}g^p d\nu=p\int_0^{\infty}\lambda^{p-1}\nu\left(\{x\in E\,:\, g(x)>\lambda\}\right)d\lambda
$$
for all $0<p<\infty$. This implies
\begin{align*}
 \int_{B_{0}\cap \{Mf>\alpha\}}f^qd\mu &\leq  (q-p)\int_{0}^{\infty}\lambda^{q-p-1}\nu(B_{0} \cap \{ Mf>\alpha \}\cap \{ Mf>\lambda \})d\lambda \\
%&=(q-p)\int_{0}^{\alpha}\lambda^{q-p-1}\nu(B_{0}\cap \{Mf>\alpha\}\cap\{Mf>\lambda\})d\lambda  + {}\
%&  \quad(q-p)\int_{\alpha}^{\infty}\lambda^{q-p-1}\nu(B_{0}\cap \{Mf>\alpha\}\cap\{Mf>\lambda\})d\lambda \\
&=(q-p)\int_{0}^{\alpha}\lambda^{q-p-1}\nu(B_{0}\cap \{Mf>\alpha\})d\lambda \\
&  \quad+ (q-p)\int_{\alpha}^{\infty}\lambda^{q-p-1}\nu(B_{0}\cap\{Mf>\lambda\})d\lambda.
\end{align*}
Replacing $d\nu=f^pd\mu$ and integrating over $\lambda$, we get
\begin{multline*}
 \int_{B_{0}\cap \{Mf>\alpha\}}f^qd\mu \leq \int_{B_{0}\cap \{Mf>\alpha \}} \alpha^{q-p}f^pd\mu  \\
+(q-p)\int_{\alpha}^{\infty}\lambda^{q-p-1}\int_{B_{0}\cap \{Mf>\lambda \}}f^pd\mu d\lambda. 
\end{multline*}
We can now use Lemma \ref{weaktype-lemma} for both integrals on the right--hand side and obtain
\begin{multline*}
%\begin{split}
 \int_{B_{0}\cap \{Mf>\alpha\}}f^qd\mu %&\leq \alpha^{q-p}\big(c\alpha^p\mu(100B_0\cap \{Mf>\alpha\}) \big) +{} \\
%&\quad (q-p)\int_{\alpha}^{\infty}\lambda^{q-p-1}\int_{B_{0}\cap \{Mf>\lambda \}}f^pd\mu d\lambda \\
%&\leq c\alpha^q\mu(100B_0\cap \{Mf>\alpha\}) + {}\\
%& \quad(q-p)\int_{\alpha}^{\infty}\lambda^{q-p-1}\big(c\lambda^p\mu(100B_0\cap \{Mf>\lambda\})\big)d\lambda \\
\leq c\alpha^q\mu(100B_0\cap \{Mf>\alpha\})\\
 + c(q-p)\int_{\alpha}^{\infty}\lambda^{q-1}\mu(100B_0\cap \{Mf>\lambda\})d\lambda. 
%\end{split}
\end{multline*}
Then by changing the order of integration we arrive at
\begin{equation*}
 \begin{split}
\int_{B_{0}\cap \{Mf>\alpha\}}f^qd\mu
&\leq c\alpha^q\mu(100B_0\cap \{Mf>\alpha\})  \\
&\quad +c(q-p)\int_{\alpha}^{\infty}\lambda^{q-1}\int_{100B_0\cap \{Mf>\lambda\}}d\mu d\lambda \\
&= c\alpha^q\mu(100B_0\cap \{Mf>\alpha\})  \\
& \quad+ c(q-p)\int_{100B_0}\int_{\alpha}^{Mf}\lambda^{q-1}d\lambda d\mu,
\end{split}
\end{equation*}
from which by integrating over $\alpha$ we conclude that
\begin{equation*}
\begin{split}
\int_{B_{0}\cap \{Mf>\alpha\}}f^qd\mu &\leq  c\alpha^q\mu(100B_0\cap \{Mf>\alpha\}) \\
& \quad +c\frac{q-p}{q}\int_{100B_0}\big((Mf)^q -\alpha^q\big) d\mu \\
&\leq c\alpha^q\mu(100B_0\cap \{Mf>\alpha\})\\
&\quad + c\frac{q-p}{q}\int_{100B_{0}}(Mf)^qd\mu.
\end{split}
\end{equation*}
\end{proof}

%Finally, before starting the proof of our main theorem we recall the following property of maximal functions.

%\begin{lem}
%\label{maximal-lemma}
%Let $f\in L_{loc}^p(X)$, $1<p<\infty$. Then there is a constant $c$ depending only on $p$ and $\C$, such that
%$$
%\int_{B}(Mf)^pd\mu \leq c\int_{B}f^pd\mu
%$$
%for all balls $B$ of $X$.
%\end{lem}
\bigskip

\begin{proof}[Proof of the Gehring lemma]
Consider a fixed ball $B_0$. Set $\alpha=\essinf_{B_0}Mf$ and let $q>p$ be an arbitrary real number for the moment. We divide the integral of $f^q$ over $B_0$ into two parts:
\begin{equation}
\label{e:g-toaprox}
\int_{B_{0}}f^q d\mu= \int_{B_{0}\cap \{Mf>\alpha\}}f^qd\mu + \int_{B_0\cap \{Mf\leq \alpha\}}f^q d\mu.
\end{equation}
The second integral in \eqref{e:g-toaprox} is easier to estimate, and we have
\begin{equation*}
 \int_{B_{0}\cap \{Mf\leq\alpha\}}f^qd\mu \leq    \int_{B_{0}\cap \{Mf\leq\alpha\}}(Mf)^qd\mu 
\leq \alpha^q\mu(100B_0\cap \{Mf\leq \alpha\}).  
\end{equation*}
It would be tempting to use Lemma \ref{technical-lemma} to the first integral in \eqref{e:g-toaprox}, but this would require $f \in L^q_{loc}(X)$. Unfortunately, that is exactly what we need to prove. The function $f$ is assumed to be locally integrable and by the reverse H\"older inequality it is also in the local $L^p$--space. Nevertheless, we can replace $f$ with the truncated function $f_i=\min\{f,i\}$. The reverse H\"older inequality \eqref{rhi}, Lemmas \ref{weaktype-lemma} and \ref{technical-lemma} as well as the preceeding analysis hold for $f_i$. In addition, $f_i\in L^q_{loc}(X)$. We continue to denote the function $f$ but remember that from now on we mean the truncated function.  

With \eqref{technical-estimate} we get now from \eqref{e:g-toaprox}
\begin{equation*}
\begin{split}
\int_{B_0}f^q d\mu &\leq c\alpha^q \mu(100B_0)\cap\{Mf>\alpha\})+ c\frac{q-p}{q}\int_{100B_{0}}(Mf)^qd\mu \\ 
& \quad+\alpha^q \mu(100B_0)\cap\{Mf\leq\alpha\}) \\
&\leq c\alpha^q \mu(100B_0) + c\frac{q-p}{q}\int_{100B_{0}}(Mf)^qd\mu
\end{split}
\end{equation*}
and furthermore
$$
\vint_{B_{0}}f^q d\mu\leq c\alpha^q + c\frac{q-p}{q}\vint_{100B_{0}}(Mf)^qd\mu.
$$
This is true for all $q>p$. Let $\eps>0$  and choose $q>p$ such that $c(q-p)/p<\eps$. Then 
\begin{equation}
\label{e:g-apu}
\vint_{B_{0}}f^q d\mu\leq c\alpha^q + \eps\vint_{100B_{0}}(Mf)^qd\mu.
\end{equation}
Now that $f=f_i$ is locally $q$--integrable, the equation \eqref{e:g-apu} gives
\begin{equation}
\label{e:g-iterati}
\vint_{B_{0}}f^q d\mu\leq c\alpha^q + c\eps\vint_{100B_{0}}f^qd\mu
\end{equation}
due to the well known boundedness theorem for maximal functions, see for example \cite{coifweis}.
We have chosen $\alpha$ such that  $\alpha \leq Mf(x)$ for $\mu$--a.e. $x$ in $B_0$. Hence
\begin{equation*}
\begin{split}
\alpha^p &= \vint_{B_0}\alpha^pd\mu \leq  \vint_{B_{0}}(Mf)^pd\mu \leq  c\vint_{100B_{0}}(Mf)^pd\mu \\
&\leq c\vint_{100B_{0}}f^pd\mu \leq c\left(\vint_{100B_{0}}fd\mu\right)^p,
\end{split}
\end{equation*}
where we use again the estimate for the Hardy--Littlewood maximal function and the reverse H\"older inequality.  Moreover
\begin{equation}
\label{e:g-alpha}
\alpha^q \leq c\left(\vint_{100B_{0}}fd\mu\right)^q. 
\end{equation}
From \eqref{e:g-iterati} and \eqref{e:g-alpha} we conclude that 
\begin{equation}
\vint_{B_{0}}f^q d\mu\leq c\eps\vint_{100B_{0}}f^qd\mu+ c\left(\vint_{100B_{0}}fd\mu\right)^q
\end{equation}
for all balls $B_0$ of $X$. If necessary, choose a smaller $\eps$ and thus also a $q$ closer to $p$ in \eqref{e:g-apu} to make Lemma \ref{l:g-2iterat} hold true. Set $k=100$ in the lemma to obtain
$$
\vint_{B_{0}}f^qd\mu  \leq c \left(\vint_{2B_{0}}fd\mu\right)^q.
$$
It remains to pass to the limit with $i\to \infty$ and the theorem follows. 
%Since $f$ satisfies the reverse H\"older inequality and the measure $\int fd\mu$ is doubling, we have 
%\begin{equation*}
%\begin{split}
%\vint_{B_0}f^qd\mu &\leq c \left(\frac{1}{\mu(2B_{0})}\int_{2B_{0}}fd\mu\right)^q \leq c \left(\frac{1}{\mu(2B_{0})}\int_{B_{0}}fd\mu\right)^q \\
%&\leq c \left(\vint_{B_{0}}fd\mu\right)^q.
%\end{split}
%\end{equation*}
\end{proof}

The following proposition implies Corollary \ref{gehring-corollary}.

\begin{pro}
Let $(X,d,\mu)$ satisfy the annular decay property and $\mu$ be doubling. Suppose that $f\in L^1_{loc}(X)$ is a non--negative function satisfying \eqref{rhi}. Then the measure induced by $f$ is doubling, i.e.
$$
\int_{2B}fd\mu\leq c\int_Bfd\mu
$$
for all balls $B$ of $X$. The constant $c$ depends only on the constant in \eqref{rhi}. 
\end{pro}
\begin{proof}
Define 
\[
\nu(U)=\int_Ufd\mu
\]
for $U\subset X$ $\mu$--measurable. Fix a ball $B$ in $X$ and let $E\subset B$ be a $\mu$--measurable set. Then
\begin{equation*}
\begin{split}
%\nu(E) &= %\int_Efd\mu = 
\int_B &f\chi_Ed\mu 
\leq \left(\int_Bf^p d\mu\right)^{1/p}\mu(E)^{1-1/p} \\
&\leq  c\left(\int_Bf d\mu\right)\mu(B)^{1/p-1}\mu(E)^{1-1/p} 
= c\nu(B)\left(\frac{\mu(E)}{\mu(B)} \right)^{1-1/p}.
\end{split}
\end{equation*}
The inequalities above follow from the H\"older and the reverse H\"older inequalities, respectively. For all $E\subset B$ this implies
\begin{equation}
\label{version1}
\frac{\nu(E)}{\nu(B)} \leq c \left(\frac{\mu(E)}{\mu(B)}\right)^{1/p'},
\end{equation}
where $p'$ is the $L^p$--conjugate exponent of $p$. Since the set $E$ in \eqref{version1} is arbitrary, we can replace it by $B\setminus E$. Therefore
\begin{equation}
\label{version2}
1-\frac{\nu(E)}{\nu(B)} =\frac{\nu(B\setminus E)}{\nu(B)} \leq c \left(\frac{\mu(B\setminus E)}{\mu(B)}\right)^{1/p'}
\end{equation}
for all $E\subset B$. If $E=(1-\delta) B$, then by choosing $0<\delta <1$ small enough we get
\begin{equation}
\label{small-alpha}
c \left(\frac{\mu(B\setminus (1-\delta) B)}{\mu(B)}\right)^{1/p'}<\frac{1}{2}
\end{equation}
by the annular decay property. It follows from \eqref{version2} and \eqref{small-alpha} that
$$
1-\frac{\nu((1-\delta) B)}{\nu(B)} < \frac{1}{2} 
$$
and hence $\nu(B)\leq 2\nu((1-\delta) B)$. We are now able to iterate this. There exists $k\in \N$ such that $(1-\delta)^k< 1/2$ and thus
\begin{equation*}
\nu(B) \leq  2v((1-\delta) B) \leq  2^k\mu((1-\delta)^kB)\leq  2^k \nu(\frac{1}{2}B)
\end{equation*}
for all balls $B$ of $X$. This proves that $\nu$ is doubling. Remark that the doubling property of $\mu$ plays no role here.
\end{proof}

For examples of spaces that satisfy the annular decay property the interested reader may see for example \cite{bu}. We mention here spaces supporting a length metric, that form a rather large subclass of such spaces. In other words, suppose that a metric space supports a doubling measure, and for all $x$ and $y$ in $X$ it holds
\[
d(x,y)=\inf\mathrm{length}(\gamma),
\]
where the infimum is taken over all rectifiable paths joining $x$ and $y$. Then the space satisfies the $\alpha$--annular decay property for constants $\alpha$ and $c$ that depend only on the doubling constant, and $0<\delta \leq 1/2$. See \cite{col-mi} for the proof.

%%%%%%%%%%%%%%%%%%%%%%%%%%%%%%%%%%%%%%%%%%%%%%%%%%%%%%%%%%%%%%%%%%%%
%                                                                  %
% A_p-weights                                                      %
%                                                                  % 
%%%%%%%%%%%%%%%%%%%%%%%%%%%%%%%%%%%%%%%%%%%%%%%%%%%%%%%%%%%%%%%%%%%%

\section{Self--improving property of \\ Muckenhoupt weights }
\label{mucken-section}

Throughout the section, let $(X,d,\mu)$ let be a metric space with a doubling measure $\mu$.

Muckenhoupt weights form a class of functions that satisfy one type of a reverse H\"older inequality. More precisely, if $1<p<\infty$, a locally integrable non--negative function $w$ is in $A_p$ if for all balls $B$ in $X$ the inequality 
$$
\left(\vint_{B} w d\mu\right)\left(\vint_{B} w^{1-p'} d\mu\right)^{p-1}\leq c_{w}
$$
holds. The constant $c_{w}$ is called the $A_p$--constant of $w$ and $1/p+1/p'=1$. Moreover, $A_1$ is the class of locally integrable non--negative functions that satisfy
$$
\vint_{B}w d\mu \leq c_{w} \essinf_{x\in B}w(x).
$$
for all balls $B$ in $X$. In this section we show that the $A_p$--condition is an open ended condition; every $w\in A_p$ is also in some $A_{p-\eps}$.

\medskip
In the following lemma number $2$ is not important and it can be replaced by any positive constant.

\begin{pro} 
\label{negative-holder}
For all locally integrable non--negative functions the inequality
\begin{equation}
\label{help5}
\left(\vint_{B} f^{-t} d\mu\right)^{-1/t}\leq  \left(\vint_{B} f^{1/2} d\mu\right)^{2}
\end{equation}
holds for all $t>0$ and all balls $B$ in $X$. 
\end{pro}

\begin{proof}
%The equation \eqref{help5} equals 
%\begin{equation*}
%\vint_{B} f^{-t} d\mu\geq  \left(\vint_{B} f^{1/2} d\mu\right)^{-2t}.
%\end{equation*}
Setting $g=f^{1/2}$ and replacing $f$ by it in \eqref{help5} gives an equivalent inequality
$$
\vint_{B} g^{-2t} d\mu\geq \left(\vint_{B}g d\mu\right)^{-2t}.
$$
This holds by the Jensen inequality since $x\mapsto x^{-2t}$ is a convex function on $\{x>0\}$.
\end{proof}

%For the beginning we shall show that every $A_p$--weight satisfies the reverse H\"older inequality. 

\begin{thm}
\label{RHI-w}
Let $1\leq p <\infty$ and $w\in A_p$. Then there exist a constant $c$ and $\eps>0$ such that
\begin{equation}
\label{improved-gehring}
\left(\vint_{B}w^{1+\eps}d\mu\right)^{1/(1+\eps)}\leq c \vint_{B}wd\mu,
\end{equation} 
where the constant depends only on the $A_p$--constant of $w$ and on the constants in the Gehring lemma.
\end{thm}

\begin{proof}
Since $A_1\subset A_p$ for all $p>1$, we can assume $p>1$. Take an arbitrary ball $B$ in $X$ and $w\in A_p$ for some $p>1$. This implies
$$
\left(\vint_{B} w d\mu\right)\leq c \left(\vint_{B} w^{1-p'} d\mu\right)^{1-p},
$$
where the right--hand side is well defined since either $w>0$ $\mu$--a.e. or $w\equiv 0$. By Proposition \ref{negative-holder} this implies
\begin{equation}
\label{help7}
\left(\vint_{B} w d\mu\right)\leq c \left(\vint_{B} w^{1/2} d\mu\right)^{2}.
\end{equation}
Now from the Gehring lemma it follows that
$$
\left(\vint_{B} w^{1+\epsilon} d\mu\right)^{1+\epsilon}\leq c \left(\vint_{B} w^{1/2} d\mu\right)^{2},
$$
where we can use the H\"older inequality and get to
\begin{equation}
\label{gehring-w}
\left(\vint_{B} w^{1+\epsilon} d\mu\right)^{1+\epsilon}\leq c \vint_{2B}w d\mu
\end{equation}
for some $\eps>0$ and constant c. To see this, in \eqref{help7} replace  $w$ by an auxiliarity function $g$ such that $w=g^2$. Then we can rewrite \eqref{help7} as
\begin{equation*}
\left(\vint_{B} g^2 d\mu\right)^{1/2}\leq c \vint_{B} g d\mu,
\end{equation*}
i.e. the reverse H\"older inequality for $g$. Gehring's lemma provides us with $\delta>0$ such that
\begin{equation*}
\left(\vint_{B} g^{2+\delta} d\mu\right)^{1/(2+\delta)}\leq c \vint_{2B} g d\mu.
\end{equation*}
This leads to \eqref{gehring-w} with $\eps =\delta/2$. Finally, we recall that a Muckenhoupt weigth induces a doubling measure, and hence \eqref{improved-gehring} follows.
\end{proof}

\begin{cor}
Let $1<p<\infty$ and $w\in A_p$. There exists $p_1<p$ such that $w\in A_{p_{1}}$.
\end{cor}

\begin{proof}
Recall that $w\in A_p$ if and only if $w^{-p'/p}\in A_{p'}$. It follows from Theorem \ref{RHI-w} that there are $\eps>0$ and a constant $c$ such that
\begin{equation}
\label{help8}
\left(\vint_{B}(w^{-p'/p})^{1+\eps}d\mu\right)^{1/(1+\eps)} \leq c \vint_{B}w^{-p'/p}d\mu.
\end{equation}
In addition,
$$
\frac{p'}{p}(1+\eps) =\frac{1+\eps}{p-1} = \frac{1}{p_1-1} = \frac{p'_1}{p_1},
$$
where $p_1=p/(1+\eps)-1/(1+\eps)+1$. Since $p>1$, $p_1<p$. The equation \eqref{help8} can now be written as
\begin{equation}
\label{help9}
\vint_{B}w^{-p'_{1}/p_{1}}d\mu \leq c \left(\vint_{B}w^{-p'/p}d\mu\right)^{1+\eps}.
\end{equation}
On the other hand, $-p'/p=1-p'$ and thus the $A_p$ condition of $w$ implies
%$$
%\left(\vint_{B}wd\mu\right)\left(\vint_{B}w^{-q/p}d\mu\right)^{p/q}\leq c.
%$$
%It follows that
\begin{equation*}
\left(\vint_{B}w^{-p'/p}d\mu\right)^{p/p'}\leq c\left(\vint_{B}wd\mu\right)^{-1}.
\end{equation*}
Raising this first to the power $p'/p$ and then to $1+\eps$,  we get
\begin{equation}
\begin{split}
\label{version4}
\left(\vint_{B}w^{-p'/p}d\mu\right)^{1+\eps}&\leq c\left(\vint_{B}wd\mu\right)^{-p'(1+\eps)/p} \\
&=    c\left(\vint_{B}wd\mu\right)^{-p'_1/p_1}.
\end{split}
\end{equation}
From \eqref{help9} and \eqref{version4} we finally conclude that
$$
\vint_{B}w^{-p'_1/p_1}d\mu \leq c\left(\vint_{B}wd\mu \right)^{-p'_1/p_1}.
$$ 
This means that $w\in A_{p_1}$, where $p_1<p$.
\end{proof}

%\appendix
%\include{fields}
%\newpage
%\include{notation}

\bibliographystyle{plain}
%\bibliographystyle{alpha}

%\label{bibliography}
\makeatletter
\addcontentsline{toc}{chapter}{\bibname}
\makeatother
\bibliography{viitteet.bib}
\nocite{*}

\end{document}